%% file: generating_snarks_springer.tex
\documentclass[pdflatex,sn-mathphys-num]{sn-jnl}

\usepackage{graphicx}%
\usepackage{multirow}%
\usepackage{amsmath,amssymb,amsfonts}%
\usepackage{amsthm}%
\usepackage{mathrsfs}%
\usepackage[title]{appendix}%
\usepackage{xcolor}%
\usepackage{textcomp}%
\usepackage{manyfoot}%
\usepackage{booktabs}%
\usepackage{algorithm}%
\usepackage{algorithmicx}%
\usepackage{algpseudocode}%
\usepackage{listings}%

\theoremstyle{thmstyleone}%
\newtheorem{lemma}{Lemma}%
\newtheorem{observation}{Observation}%

\theoremstyle{thmstylethree}%

\usepackage{geometry} 
\usepackage{amsmath}
\usepackage{graphicx} 
\usepackage{color}

\usepackage{subfig}
\usepackage{tikz}
\usepackage{tkz-euclide}

\usepackage{booktabs} 
\usepackage{array} 
\usepackage{paralist} 
\usepackage{verbatim} 
\usepackage{subfig} 
\usepackage{tikz}
\usepackage{tkz-euclide}
\usepackage{fancyhdr} 

\usepackage{amsfonts}

\usepackage{amsmath}

\begin{document}



\title[Generating Snarks]{ Algorithms for the Generation of Snarks }
\author[]{\fnm{Gunnar}\sur{Brinkmann}}\email{gunnar.brinkmann@ugent.be}
\author[]{\fnm{Steven}\sur{Van Overberghe}}\email{steven.vanoverberghe@ugent.be}

\affil[]{\orgdiv{Department of Mathematics, Computer Science, and Statistics}, \orgname{Ghent University}, \orgaddress{\street{Krijgslaan 299 - S9}, \city{9000 Ghent}, \country{Belgium}}}

\abstract { The essential requirement for a cubic graph to be called a snark is that it can not be edge-coloured with three colours. To avoid trivial cases, varying 
  restrictions on the connectivity are imposed. Snarks are not only interesting in themselves, but also a
  valuable test field for conjectures about graphs that are not snarks and sometimes not even cubic.  For many important open problems in graph
  theory it is proven that minimal counterexamples would be snarks.\\
 
  We give two new algorithms for the generation of snarks and results of computer programs implementing these algorithms. One algorithm is for snarks with girth
  exactly 4 and is used for generating complete lists of girth 4 snarks on up to 40 vertices.  The second algorithm lists snarks with girth at least 5 and is used
  for generating complete lists of such snarks on up to 38 vertices. We also give complete lists of strong snarks (in the terminology of Jaeger)
  on up to 40 vertices.\\
}

\keywords{graph, snark, structure enumeration}

\maketitle

\section{Introduction}

 Information about conjectures for which it was proven that snarks are possible minimal counterexamples and basic results about snarks and reducibility can
 e.g. be found in \cite{snarkapp} or \cite{snarklist3}. In this paper we will just focus on the algorithmic aspects of generating snarks. We will follow the
 tradition to call graphs with maximum degree $3$ and chromatic index 3 {\em colourable} (or {\em class 1}) and graphs with maximum degree $3$ and chromatic
 index 4 {\em uncolourable} (or {\em class 2}). The definition of when an uncolourable cubic graph is called a snark varies a lot. The initial definition by
 Gardner \cite{gardner_76} just required a snark to be a simple cubic bridgeless uncolourable graph. So not allowing bridges was the only requirement to avoid
 {\em trivial} cases.  Later various definitions with stronger requirements on girth and connectivity were also used, depending on different notions of what
 must be considered trivial and therefore excluded.  Some earlier papers called cyclically 4-connected cubic class 2 graphs with girth 4 {\em weak snarks}. Already before,
 Jaeger \cite{Jaeger85} introduced the notion of {\em strong snarks}, which can also have girth exactly 4. This implies the existence of {\em
   strong weak snarks}. To avoid this self-contradictory term, we will use the term {\em girth 4 snarks} (or short {\em g4-snarks}) for cyclically 4-connected cubic graphs with
 chromatic index 4 and girth exactly 4. We will refer to cubic graphs with chromatic index 4 and girth at least 5 explicitly as {\em proper snarks}. The
 term {\em snark} refers to the union of these classes.

 Until $1975$, only $4$ proper snarks were known -- the {\em Petersen graph}, the {\em Blanu{\v s}a snark} (later, a second, closely related, snark called Blanu{\v s}a snark was discovered),
 the {\em Descartes snark}, and the {\em Szekeres snark} (\cite{isaacs_75}). In his $1975$ paper Isaacs then constructs an infinite family of
 snarks, but in spite of the infinite number, they form no convincing set for conjecture testing as they all share a common very specific structure. Since then, more
 constructions of snarks have been proposed, but again with the same intrinsic problem of exhibiting a very special structure. To this end, for conjecture testing, complete lists
 of snarks for as many vertices as possible are desirable.

 A recursive construction method, that is: a method constructing {\bf all} snarks on $n$ vertices from snarks with fewer than $n$ vertices by simple constructions, would be ideal, but
 no promising ideas or approaches in this direction exist. So far the most efficient programs to generate snarks were based on efficient programs for the generation of all
 cubic graphs with girth at least 4 or 5. These programs were then equipped with routines guaranteeing cyclic connectivity at least 4 and chromatic index 4 for
 the output graphs.

 When generating all cubic graphs with girth 4 or 5, cyclic connectivity at least 4 does not restrict the set of graphs much. On 28 vertices, for example, more
 than $88.5\%$ of the cubic graphs with girth at least 4 are also cyclically 4-connected. This ratio even seems to be increasing with the number of vertices.
 On the other hand, fewer than $0.09\%$ of all cubic graphs on 28 vertices with girth at least 4 have chromatic index 4 and the ratio seems to be decreasing. So
 for graph generators designed to generate {\bf all} cubic graphs, a method to detect early whether the final graph constructed will be colourable is crucial to
 be able to generate snarks on $n$ vertices much faster than the superset of cubic graphs without restrictions on the chromatic index.

 Nevertheless, the first program to generate relatively large sets of snarks, {\em minibaum} \cite{minibaum}, developed in 1992, did nothing like that: it
 just generated all cubic graphs and filtered the output for cyclically 4-connected graphs with chromatic index 4. Essentially the same holds for the later
 program described in \cite{snarklist3}, for which no running times have been given.  In 2011 the program {\em snarkhunter} was developed
 \cite{snarkapp}\cite{snarkhunter}.  It was approximately 14 times faster than minibaum for generating proper snarks and almost 30 times faster for generating all snarks. The
 key was that it was the first program that was able to efficiently include bounding criteria avoiding colourable graphs already before the graph was
 completed. This new program made it possible to generate all snarks on up to 36 vertices. The proper snarks
 alone already took 73 CPU years and due to the fast growth of computing times, the next step would have needed far more than 1000 core years at that time. Tests
 for 24 up to 32 vertices 
 on the local HPC infrastructure we used for our computations (processor: AMD EPYC 7552 (Rome @ 2.2 GHz)), suggest that 
 nowadays about $840$ CPU years would be necessary for the proper snarks on 38 vertices.
 For all -- that is: girth 4 as well as proper -- snarks the time consumption would be
 approximately $1 250$ CPU years. To this end a much faster algorithm was needed in order to be able to construct complete lists of snarks for the next step: 38 vertices. 

 
\begin{figure}[tb]
	\centering
	\includegraphics[width=0.75\textwidth]{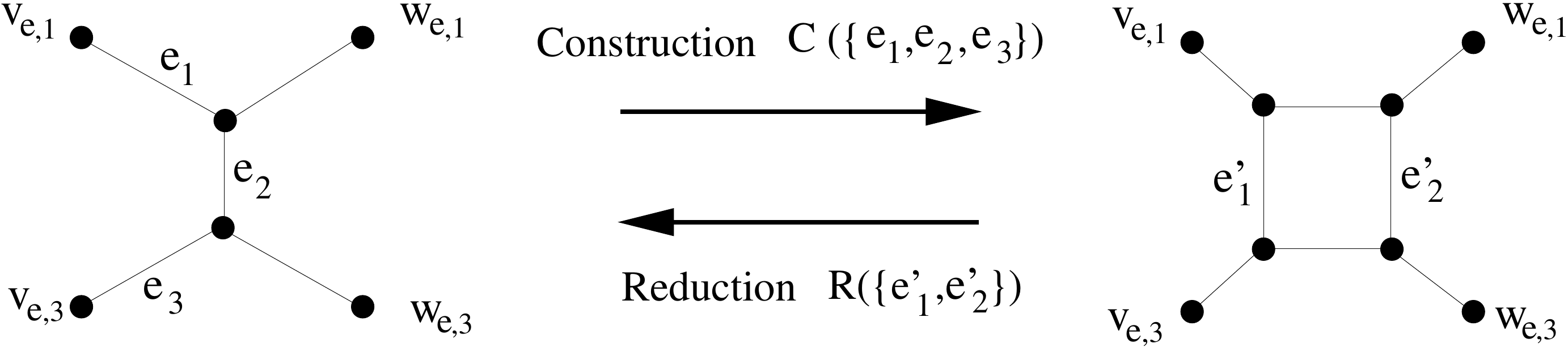}
	\caption{Construction and reduction operations used for girth 4 snarks. The construction operation applied to the path $e_1,e_2,e_3$ is denoted as $C(\{e_1,e_2,e_3\})$.
          The reduction operation applied to the pair $e'_1,e'_2$ of edges that are opposite in a 4-cycle, is denoted as $R(\{e'_1,e'_2\})$. }
	\label{fig:weakconstruct}
\end{figure}

 \section{Girth 4 Snarks -- the algorithm for the program {\em tetration}}\label{sec:weak}

 We call the construction operation $C()$ displayed in Figure~\ref{fig:weakconstruct} an {\em edge-doubling operation}. Formally {\em doubling edge $e_2$ in
 the path $e_1,e_2,e_3$} is defined as follows: the edges $e_1$ and $e_3$ are subdivided and the new vertices of degree 2 are
 connected. In total there are 4 ways to apply an edge doubling operation for a central edge $e_2$, but it is easy to see that they form two pairs of operations
 that give isomorphic results.

 For the inverse operation, edges on opposite sides of a 4-cycle, like $e'_1,e'_2$ in Figure~\ref{fig:weakconstruct}, play an important role. They will be referred to as {\em edge
   pairs}. 

We use the following two lemmas that are also consequences of Theorem~1 in \cite{decosnarks}, resp.\ Corollary 11 in \cite{removableedges4cyc}, but as the
formulations in these papers are different and sometimes more general, we will give short proofs of the specific results we need.

\begin{figure}[tb]
	\centering
	\includegraphics[width=0.65\textwidth]{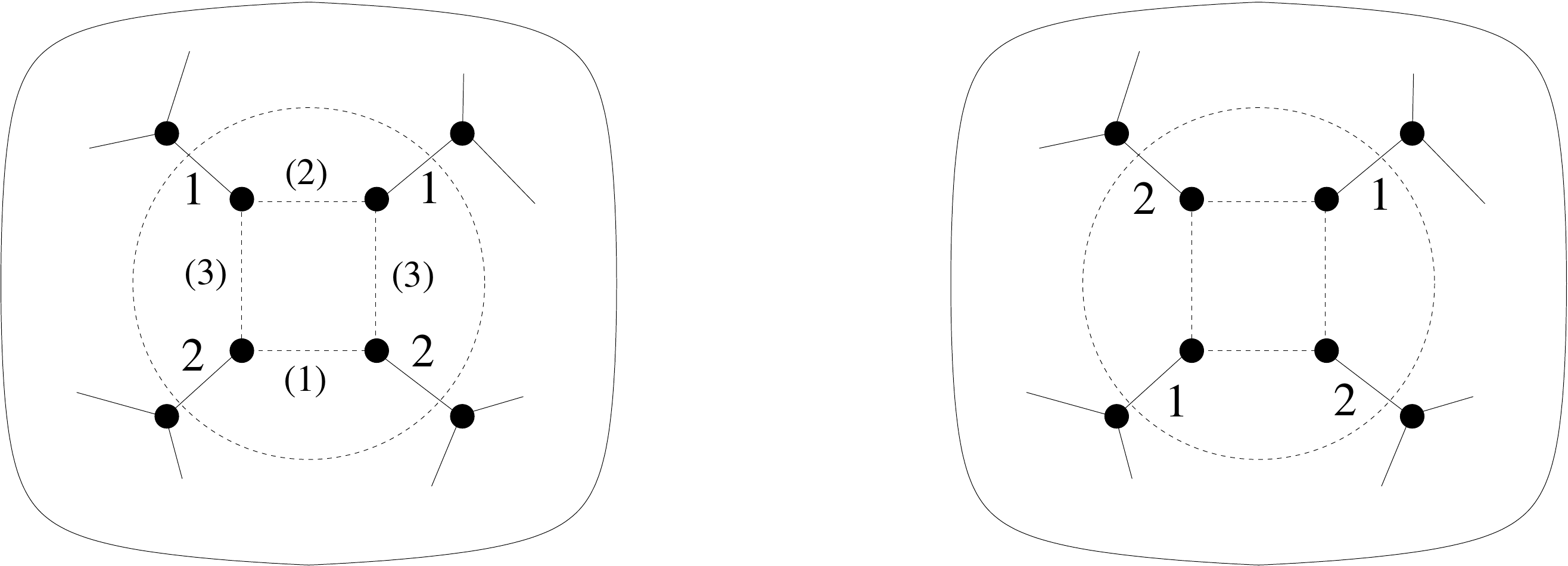}
	\caption{The colours of the edges adjacent to the 4-cycle in the proof of Lemma~\ref{lem:deco}.}
	\label{fig:colours}
\end{figure}

\begin{lemma}\label{lem:deco}
  (Cameron, Chetwynd, Watkins \cite{decosnarks})
\nopagebreak[3]
A cubic graph $G$ containing a 4-cycle $Z$ is class 2 if and only if the graph $G'=G-Z$, in which the vertices of the 4-cycle are removed, is class 2.
  
\end{lemma}

\begin{proof}
  $\Leftarrow$: This direction is obvious as adding vertices and edges can not decrease the number of colours that is necessary to colour the edges.

  $\Rightarrow$: Let $G$ be class 2 and assume that $G'$ can be edge-coloured with 3 colours. Then also the graph $G''$, in which only the edges of $Z$ were
  removed, can be edge-coloured with 3 colours.  Due to the Parity Lemma, each colour must occur an even number of times on the 4 edges connecting the vertices of $Z$ with
  the rest of the graph.  If one colour occurs 4 times, $Z$ can be coloured with the remaining two colours -- a contradiction. So w.l.o.g.\ one of the two
  situations in Figure~\ref{fig:colours} must occur. In the situation on the left, $Z$ can be coloured in the way given by the colours in brackets. In the
  situation on the right, $Z$ can not be coloured, so look at the Kempe-chain with colours 1 and 2 starting at one of the edges coloured 1. The chain must end
  in one of the other edges coloured 1 or 2.  If it is the edge coloured 1, switching colours in this chain results in the situation with four times colour
  2. If it ends in an edge coloured 2, switching colours results in the situation on the left, so in each case $G$ would be colourable -- a contradiction.

  \end{proof}

The following trivial but helpful observation will be used here and also later for {\em strong snarks}, so we formulate it explicitly for later reference:

\begin{observation}\label{obs:reduce}
  \begin{itemize}
  \item Subdividing edges in a graph does not change the cyclic connectivity.
    
  \item If an edge $e$ is removed from a cyclically $k$-connected cubic graph $G$ and the result is the graph $G_e$, then $G_e$ has a cyclic $(k-1)$-cut $C$
  if and only if $C \cup \{e\}$  is a cyclic $k$-cut in $G$.
\end{itemize}
\end{observation}

\begin{lemma}\label{lem:reduce}
  (Andersen, Fleischner, Jackson \cite{removableedges4cyc})

  Let $G$ be a cyclically 4-connected cubic graph on at least 10 vertices with a 4-cycle $e_1,e_2,e_3,e_4$. Then at least one of the graphs obtained by applying
 $R(\{e_1,e_3\})$ and $R(\{e_2,e_4\})$ is cyclically 4-connected.

\end{lemma}

\begin{proof}

  We will use the notation from Figure~\ref{fig:weakproof}. 

  Assume that the result $G'$ of applying $R(\{e_1,e_3\})$ to $G$ has a cyclic 3-cut. Then the graph obtained from $G$ by removing $e_3$ -- which is
  isomorphic to $G'$ with two edges subdivided -- has a cyclic 3-cut and by Observation~\ref{obs:reduce} the edge $e_3$ is in a cyclic 4-cut $C_0$ of $G$.  As
  in minimal cyclic cuts in cubic graphs the edges do not share vertices and as the endpoints of $e_1$ are in different components of $G-C_0$, also $e_1$ is in
  $C_0$ and we have one component $C_{3,4}$ containing $v_3$ and $v_4$ and one component $C_{1,2}$ containing $v_1$ and $v_2$. Removing also the vertices of $Z$,
  both resulting sets $C^-_{1,2}$ and $C^-_{3,4}$ have 4 edges leaving the set in $G$ with edges between the parts counted as leaving both.

  Assuming that also the result of $R\{(e_2,e_4\})$ has a cyclic 3-edge cut, in the same way we get sets $C^-_{1,4}$ and $C^-_{2,3}$ that are left by 4 edges
  each -- two of those four the edges adjacent to $Z$ that also occur for $C^-_{1,2}$ and $C^-_{3,4}$. Let now $C_1=C^-_{1,2}\cap C^-_{1,4}$ and analogously for $C_2,C_3$, and $C_4$.
  Then $C_1,\dots,C_4$ is a partition of the vertex set of $G-Z$ and in total we have $12$ edges leaving $C_1,\dots,C_4$.

  As in the complement of each $C_i$ in $G$ there is the cycle $Z$ and as edge-cuts with fewer than 3 edges are always cyclic cuts, each $C_i$ must be left by exactly 3 edges which form a
  non-cyclic cut. So each $C_i$ consists of a single vertex, which implies that $G$ has 8 vertices -- a contradiction.

\end{proof}

\begin{figure}[tb]
	\centering
	\includegraphics[width=0.45\textwidth]{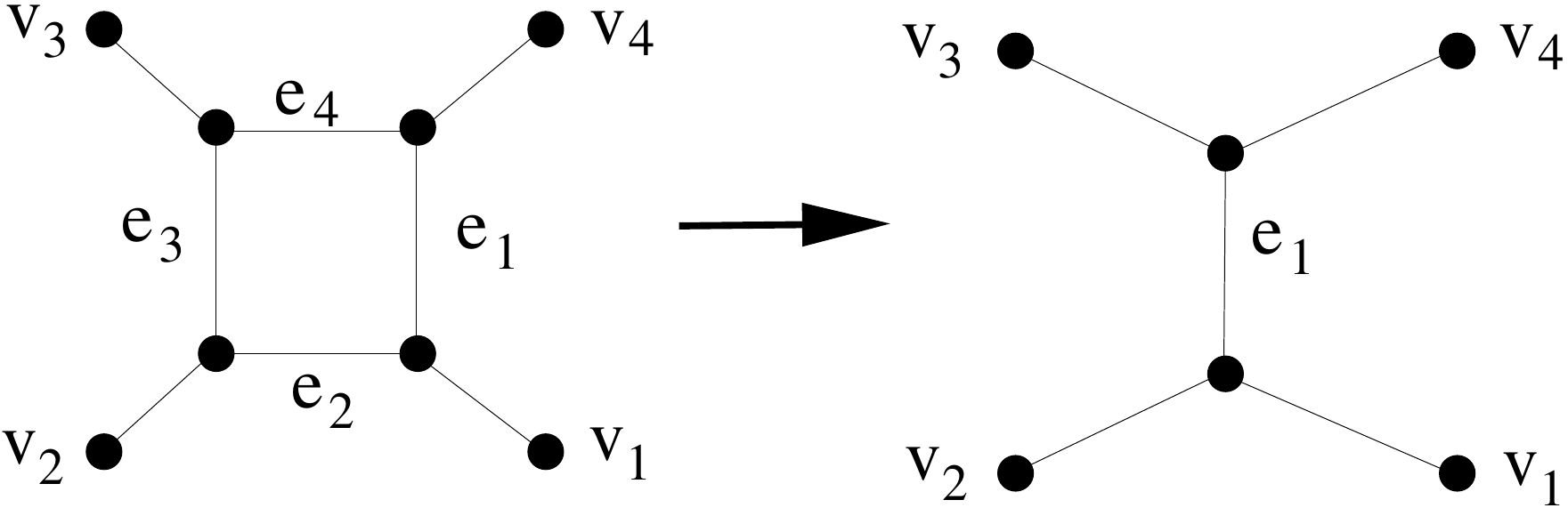}
	\caption{The notation for vertices and edges for the proof of Lemma~\ref{lem:reduce}.}
	\label{fig:weakproof}
\end{figure}

The following lemma essentially says that all g4-snarks are direct or indirect descendants of proper snarks by the edge doubling operation. A similar result for
proper snarks with an operation subdividing and connecting edges at larger distances is not possible. Although all 38
proper snarks on 24 vertices can be obtained from snarks on 22 vertices by subdividing edges and connecting the new vertices, for all other vertex numbers for which snarks exist, there are proper snarks that can not be obtained this way as follows from the -- much stronger -- theorems A and B in \cite{irreducible_snarks}.

\begin{lemma}\label{lem:weak}
   ~\\
   \begin{description}

   \item[a.)] The result of applying edge-doubling operations in each possible way to the (g4- or proper) snarks on $n$ vertices gives a superset $M$ of the class of
   g4-snarks on $n+2$ vertices.

 \item[b.)] All graphs in $M$ are cubic graphs with girth and cyclic connectivity exactly 4, but they are not necessarily class 2. Nevertheless, when applying the edge doubling operation
   in the four possible ways to an edge $e$, either all results are class 2 or none of them.

   \end{description}

   \end{lemma}

 \begin{proof}

   {\bf a.)}  To show that each g4-snark on $n+2$ vertices can be obtained this way, we prove that each such snark can be reduced to a snark on $n$ vertices by
   the inverse of an edge-doubling operation:

   Let $Z$ be a 4-cycle in a g4-snark $G$ on $n+2$ vertices. Then due to Lemma~\ref{lem:deco} even the subgraph $G-Z$ of the reduced graph is class 2, so it just remains to
   be shown that at least one of the reductions also produces a cyclically 4-connected graph, but this is exactly the statement of Lemma~\ref{lem:reduce}.

  \medskip

 {\bf b.)}  The fact that after the construction applied to a cubic cyclically 4-connected graph $G$, the result $G'$ is a cubic graph with girth exactly 4 is
 obvious. To prove that it is also cyclically 4-connected, assume that edges $e_1,e_3$ are subdivided. The result after subdivision is still cyclically 4-connected due to
 Observation~\ref{obs:reduce}. Inserting a new edge can reduce the cyclic connectivity, if former non-cyclic cuts become cyclic, but as the non-cyclic subgraph with the smallest
 number of outgoing edges that can become cyclic by the new edge already has 4 outgoing edges, $G'$ is cyclically 4-connected.

 The fact that the result of the operation is not necessarily class 2, can already be deduced from the non-existence of snarks on 12 vertices -- so applying the
 edge-doubling operation to an edge of the Petersen graph gives a colourable graph. Lemma~\ref{lem:deco} on the other hand implies that if the result of applying the
 edge-doubling operation to an edge $e$ is class 2, in fact already the graph obtained by removing $e$ is class 2, so no matter how it is extended to a cubic graph -- and especially if
 by one of the four ways in which the operation can be applied -- the result will always be class 2.

 \end{proof}

The most simple and na\"ive way to use Lemma~\ref{lem:weak} to generate g4-snarks on $n$ vertices from snarks on $n-2$ vertices would be to apply the edge doubling
operation in all possible ways to snarks on $n-2$ vertices, filter the result for class 2 graphs and finally use standard programs for isomorphism rejection on the
large lists of graphs generated.
 
Although there are -- for a not too small number $n$ of vertices -- many more g4-snarks than proper snarks on $n$ vertices (see Table~\ref{tab:results}), even
this simple method would take far less time than the generation of the proper snarks on $n$ vertices. Nevertheless this should not be taken as an excuse to not
develop a faster method that can also be distributed on several computers and can generate g4-snarks on $n+2$ vertices if the lists of
proper snarks on $n$ vertices are available. We use the following method implementing McKay's {\em canonical construction path} method \cite{McK96} for this application.  The
reduction operation $R(\{e'_1,e'_2\})$ corresponding to $C(\{e_1,e_2,e_3\})$, is the one with $e'_1$ the edge connecting the two subdivided edges and with
$e'_2=e_2$. This corresponding operation has the property that when applied, the original graph is restored, so we denote
it as $C^{-1}(\{e_1,e_2,e_3\})$.  The method can now be described as follows:

\begin{description}
\item[Definition:] Let $C(\{e_1,e_2,e_3\})$ and $C(\{f_1,f_2,f_3\})$ be construction operations with corresponding reduction operations $R(\{e'_1,e'_2\})$, resp.\  $R(\{f'_1,f'_2\})$
  and let $G_C$, resp.\ $G'_C$ be the results of the operations. If there is an isomorphism
  from $G_C$ to $G'_C$ mapping $\{e'_1,e'_2\}$ onto $\{f'_1,f'_2\}$, then define $C(\{e_1,e_2,e_3\})$ and $C(\{f_1,f_2,f_3\})$ as equivalent.

\item[Step 1:] When applying edge doubling operations to a graph, then compute equivalence classes of operations and apply exactly one operation in each class.
  Lemma~\ref{lem:centraledge} will give a criterion to decide on equivalence without computing the extended graphs.

\item[Definition:] For a cyclically 4-connected graph $G$ with girth exactly 4, let $E_2(G)$ be the set of all edge pairs $\{e_1,e_2\}$, so that when applying
  $R(\{e_1,e_2\})$, the result is cyclically 4-connected.  Define a {\em canonical choice function} choosing an orbit of the automorphism group on $E_2(G)$ --
  that is: If $G,G'$ are isomorphic graphs, then there must be an isomorphism mapping the two chosen orbits onto each other.

 \item[Step 2:]  If a graph $G'$ is constructed by an operation $C(\{e_1,e_2,e_3\})$, then accept the graph if and only if it is class 2 and the two defining edges of
    $C^{-1}(\{e_1,e_2,e_3\})$ are an edge pair in the canonical orbit.

\end{description}

\begin{lemma}\label{lem:centraledge}

  Two construction operations $C(\{e_1,e_2,e_3\})$ and $C(\{f_1,f_2,f_3\})$ of a cubic graph $G=(V,E)$ with girth 4 are equivalent, if and only if 
  there is an automorphism $\phi()$ of $G$ mapping $e_2$ to $f_2$
  so that $|\{\phi(e_1),\phi(e_3)\} \cap \{f_1,f_3\}| \in \{0,2\}$.

\end{lemma}

\begin{proof}

  {\bf $\Leftarrow$:}\\
  The result is obvious if $\{e_1,e_3\}=\{f_1,f_3\}$, so
  assume an automorphism $\phi()$ of $G$ mapping $e_2$ to $f_2$ so that $|\{\phi(e_1),\phi(e_3)\} \cap \{f_1,f_3\}| \in \{0,2\}$ to be given.\\

  If $\phi()$ is the identity, then it is easy to check that the results of applying $C(\{e_1,e_2,e_3\})$ and $C(\{e'_1,e_2,e'_3\})$ give isomorphic results
  if $\{e_1,e_3\}= \{e'_1,e'_3\}$ or
  $\{e_1,e_3\}\cap  \{e'_1,e'_3\}=\emptyset$. On the set  $V\setminus e_2$, the isomorphism between the extended graphs is in fact still the identity.
  Furthermore it is obvious that in that case the sets of defining edges for the inverse operation are mapped onto each other, so that
  $C(\{e_1,e_2,e_3\})$ and $C(\{e'_1,e_2,e'_3\})$ are equivalent.

  If we now have a nontrivial automorphism $\phi()$, then $C(\{e_1,e_2,e_3\})$ and $C(\{\phi(e_1),\phi(e_2),\phi(e_3)\})$ give results that are isomorphic by
  the unique isomorphism extending $\phi()|_{V\setminus e_2}$, and due to the first part also $C(\{f_1,\phi(e_2),f_3\})$ with $f_1,\phi(e_2),f_3$ a path in $G$
  with $|\{\phi(e_1),\phi(e_3)\} \cap \{f_1,f_3\}| \in \{0,2\}$.

  \bigskip

  {\bf $\Rightarrow$:}\\
  Assume now that two construction operations $C(\{e_1,e_2,e_3\})$ and $C(\{f_1,f_2,f_3\})$ of a graph $G=(V,E)$ are equivalent.\\

  So there is an isomorphism $\phi'()$ mapping the results of $C(\{e_1,e_2,e_3\})$ and $C(\{f_1,f_2,f_3\})$ applied to $G$ onto each other with the property
  that $\phi'(\{e'_1,e'_2\})=\{f'_1,f'_2\}$ with $\{e'_1,e'_2\}$, resp.\ $\{f'_1,f'_2\}$ the edge pairs defining the corresponding reductions.  
  For the automorphism of $G$, we choose $\phi(v)=\phi'(v)$ for all vertices not in $e_2$. With the notation of
  Figure~\ref{fig:weakconstruct} (and analogously for $f_1,f_2,f_3$) we have that  $\phi'(\{v_{e,1},w_{e,1}\})=\{v_{f,1},w_{f,1}\}$ and   $\phi'(\{v_{e,3},w_{e,3}\})=\{v_{f,3},w_{f,3}\}$
  or $\phi'(\{v_{e,1},w_{e,1}\})=\{v_{f,3},w_{f,3}\}$ and   $\phi'(\{v_{e,3},w_{e,3}\})=\{v_{f,1},w_{f,1}\}$, as these pairs of vertices are connected through the new 4-cycle by a path of length 3 not containing an edge
  of the pair defining the reduction. In both cases we can extend the partial automorphism $\phi()$ to $e_2$.

  On the other hand we can use that also paths of length 3 with the middle edge an edge in the pair defining the reduction $C^{-1}(\{e_1,e_2,e_3\})$ must be mapped
  onto paths with the same property for $C^{-1}(\{f_1,f_2,f_3\})$, which implies $\phi'(\{v_{e,1},v_{e,3}\})=\{v_{f,1},v_{f,3}\}$ or
  $\phi'(\{v_{e,1},v_{e,3}\})=\{w_{f,1},w_{f,3}\}$.  But as the condition $|\{\phi(e_1),\phi(e_3)\} \cap
  \{f_1,f_3\}| \in \{0,2\}$ we have to prove is equivalent to $|\phi'(\{v_{e,1},v_{e,3}\})\cap \{v_{f,1},v_{f,3}\}|\in \{0,2\}$, we are done.

  \end{proof}

\begin{lemma}\label{lem:only1}

  The algorithm described above applied to one graph of each isomorphism class of (g4- or proper) snarks on $n-2$ vertices,
  accepts exactly one graph from each isomorphism class of g4-snarks on $n$ vertices.

\end{lemma}

\begin{proof}

  Let $G$ be a g4-snark on $n$ vertices and $\{e_1,e_2\}$ an edge pair in the canonical orbit. Then $R(\{e_1,e_2\})$ gives a snark on $n-2$ vertices and there is a path
  (in fact even two paths) $e'_1,e'_2,e'_3$, so that $C^{-1}(\{e'_1,e'_2,e'_3\})=R(\{e_1,e_2\})$. As we apply an operation from the equivalence class of $C(\{e'_1,e'_2,e'_3\})$,
  we obtain a graph isomorphic to $G$, which will be accepted, so we get a graph isomorphic to $G$.

  It remains to be shown that we do not get two isomorphic copies of the same graph.

  Assume that two isomorphic graphs $G,G'$ are accepted and that they are constructed from $G_a$, resp.\ $G'_a$. So there is an isomorphism $\phi()$ mapping $G$
  onto $G'$ and the canonical orbits onto each other. Applying a canonical reduction to any edge pair in a canonical orbit in $G$ or $G'$, we get -- up to isomorphisms --
  the same graph, so $G_a$ and $G'_a$ are isomorphic and as we apply the construction operation to only one graph from each isomorphism class, we have that $G_a=G'_a$
  and by definition the two operations are equivalent, so that only one was applied -- a contradiction.

\end{proof}

To implement this algorithm, we used {\em nauty} \cite{nauty2} -- an often used and very well tested
computer program that can compute the automorphism group of a graph and also give a {\em
  canonical labelling of the vertices}.  A canonical labelling of the vertices is a function that assigns numbers to the vertices, so that any two isomorphic
graphs are transformed into the same labelled graph. The automorphism group can be used to compute the orbits of edges for the graphs to which the construction is applied.
A canonical labelling of the vertices can not only be used to decide on a canonical vertex, but also e.g. on a {\em canonical edge pair} (e.g. the lexicographically smallest one
where the reduction produces a cyclically 4-connected graph). As we also know the automorphism group, the {\em canonical orbit} can be chosen as the orbit of the chosen edge pair.
Often, some fast heuristics can be used to decide on the orbit of the canonical edge pair even without applying nauty -- e.g. if the graph has only one 4-gon.

In order to determine which edges must be doubled to get a larger class 2 graph, we use the following lemma:

\begin{lemma}\label{lem:double}

  Let $G$ be a snark and $e$ be an edge of $G$. The result $G'$ of doubling $e$ is colourable if and only if $G$ has a 2-factor with exactly 2 odd cycles $C_1,C_2$ and $e$
  contains a vertex of both odd cycles.

\end{lemma}

\begin{figure}[tb]
	\centering
	\includegraphics[width=0.16\textwidth]{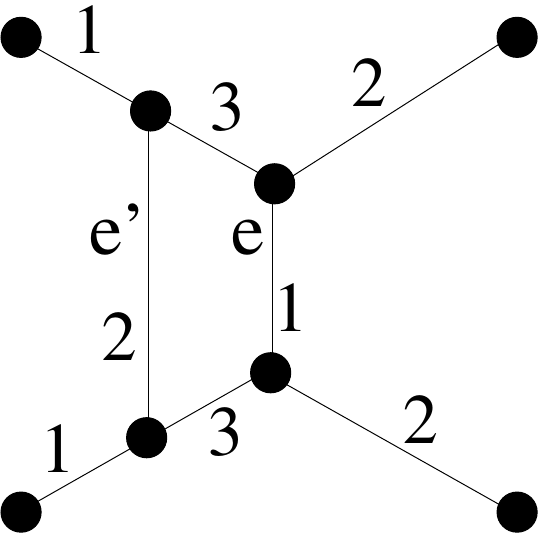}
	\caption{An illustration of an argument in the proof of Lemma~\ref{lem:double}.}
	\label{fig:lem}
\end{figure}

\begin{proof}

  $\Leftarrow$: If $e$ contains a vertex of both, $C_1,C_2$, then $e$ is itself not on that cycle, so one edge that will be subdivided is on $C_1$ and one on $C_2$.
  When the edges are subdivided, $C_1$ and $C_2$ become even cycles, so we have an {\em even 2-factor} -- a 2-factor with only even cycles -- so the graph is colourable.

  $\Rightarrow$: If $G'$ is colourable, the two colours different from that of edge $e$ define an even 2-factor $F$ and the endpoints of $e$ neighbour those of
  the new edge $e'$ on $F$. If $e'$ is not contained in $F$, the endpoints lie on different cycles of $F$, as otherwise removing $e$ and the subdivisions of the
  edges would give an even 2-factor of $G$ -- a contradiction.  So removing $e'$ and the subdividing vertices gives a 2-factor with two odd cycles -- each
  containing one of the two formerly subdivided edges and therefore a vertex of $e$.

  Assume now that $e'$ lies on $F$, so it has a different colour than $e$ and the situation is up to a permutation of colours as in Figure~\ref{fig:lem}. In this case
  removing $e'$ and the adjacent edges coloured 3 from the 2-factor and adding edge $e$ would give an even 2-factor of $G$ -- a contradiction.

\end{proof}

We compute all 2-factors with exactly two odd cycles for $G$, and for each of them mark all edges between those odd cycles. Afterwards we only double unmarked edges, so that all larger graphs
are class 2 and do not have to be tested.

\section{Proper Snarks -- the algorithm for the program {\em minisnark}}

Generating proper snarks is the computationally far more challenging and time consuming part. Already former generators like {\em minibaum} \cite{minibaum} or {\em
  snarkhunter} \cite{snarkhunter} used that cubic graphs are class 1 if and only if they have an even 2-factor, but minibaum used this equivalence only for a
final test before output.  The construction method used for {\em minibaum} works by inserting edges joining already existing vertices.  {\em In principal} this
has the nice property that if during the construction an even 2-factor is detected, the construction can stop and backtrack, as all successors on the same
number of vertices will also have the same even 2-factor and therefore be class 1. Unfortunately the method of isomorphism rejection used in minibaum requires a
breadth first order in which the edges are inserted, so that often the graph does not have any 2-factors at all -- let alone even ones -- before
by far most of the edges are inserted and the possible detection of an even 2-factor comes much too late to speed up the generation.

The most important construction operation in snarkhunter is similar to the construction operation in Figure~\ref{fig:weakconstruct} (see also Section~\ref{sec:strong}):
two edges (not necessarily adjacent to
the same third edge like in Section~\ref{sec:weak}) are subdivided and the new centers are connected. This can destroy even 2-factors, as
subdividing one edge in a cycle in an even 2-factor makes it an odd cycle. What seems to be a disadvantage can also be used: the strategy of snarkhunter when
generating snarks on $n$ vertices is to compute even 2-factors for the graphs on $n-2$ vertices and only apply construction operations that destroy those
2-factors. This does not guarantee that only class 2 graphs are generated, as not all even 2-factors for the graph on n-2 vertices are computed and as in fact
the last edge added can even create a new even 2-factor. Nevertheless it dramatically increases the fraction of class 2 graphs among all graphs generated.

The present approach was inspired by the one in \cite{permutationsnarks} generating permutation snarks. Permutation snarks are defined as snarks that have a
2-factor of two induced (so chordless) cycles of length $|V|/2$ -- and these special chordless cycles are used for the construction. Unfortunately we can not
rely on the existence of an induced 2-factor, but due to Petersen's matching theorem, each bridgeless cubic graph has a perfect matching, so also an (in general
not induced) 2-factor. To this end, we can start our construction with a 2-factor and then insert additional edges to make the graph cubic. In order to have an
independent approach from \cite{snarkhunter}, we do not use the result from \cite{oddness4}, which states that for fewer than 44 vertices, each snark has a
2-factor with only 2 odd cycles. This would only give a relatively small speedup that does not justify the dependency.

For a given ordered set $F=\{C_1,C_2,\dots ,C_{k+m}\}$ of cycles where $C_1,\dots ,C_k$ have odd lengths $l_1\le l_2\le \dots \le l_k$, and $C_{k+1},\dots ,C_{k+m}$
have even lengths
$l_{k+1}\le l_{k+2}\le \dots \le l_{k+m}$, we call $(l_1, \dots ,l_{k+m})$ the {\em type} of the set of cycles and in case the cycles form a 2-factor, the type of the 2-factor.
A {\em legal labelling} of $F$ is a labelling of the vertices with $1,\dots ,\sum_{i=1}^{k+m}l_i$, so that for $i<j$ the labels of
vertices in $C_i$ are smaller than labels of vertices in $C_j$ and where each cycle is labelled in the order of the cycle. An example for cycle lengths
$5,7,4,6$ is given in Figure~\ref{fig:legallabel}.  We will not make a distinction between {\em vertex labels} and the (pairwise distinct) labels in fact being
the vertices, so that a labelling can also be considered as an isomorphism onto a graph where the labels are the vertices. Whenever we refer to an {\em ordered
  2-factor}, we assume the ordering to be as described above.
 
\begin{figure}[tb]
	\centering
	\includegraphics[width=0.45\textwidth]{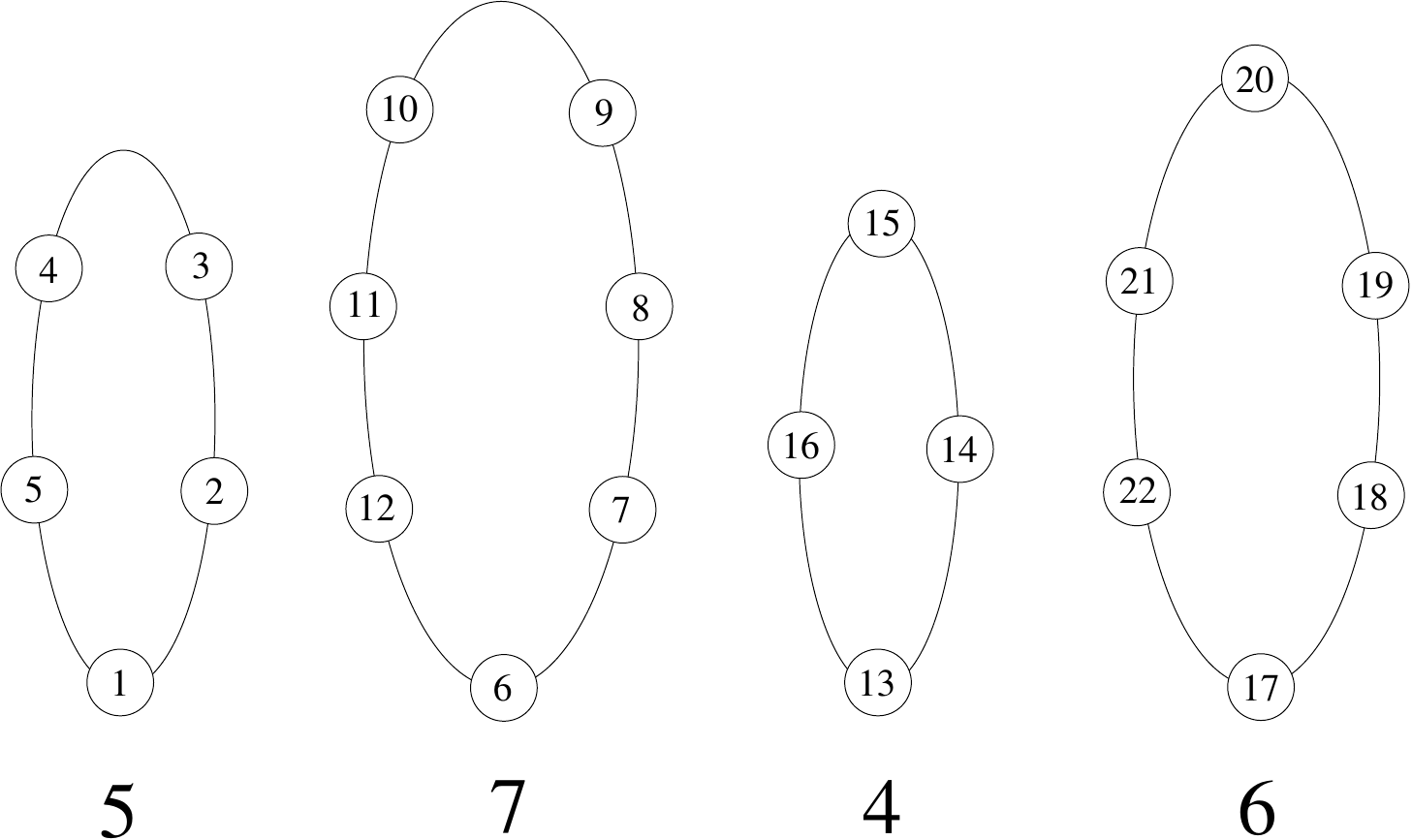}
	\caption{A legal labelling for cycles with lengths $5,7,4,6$.}
	\label{fig:legallabel}
\end{figure}

The (very) basic description of the construction algorithm for $n$ vertices and girth $g$ is as follows. We say that an edge $\{v,w\}$ is lexicographically
smaller than a disjoint edge $\{v',w'\}$, if $\min\{v,w\} < \min\{v',w'\} $.

Assume that lower bounds $g$ for the girth and $n$ for the number of vertices are given.

\begin{description}

\item[a.)] Enumerate all ordered sets of 2-factors with at least 2 odd cycles, where each cycle has length at least $g$ and where the sum of all cycle lengths is $n$.

\item [b.)] For each such 2-factor $F$, construct the initial 2-regular graph consisting of exactly the cycles described in the 2-factor and labelled with a legal labelling.

\item[c.)] For each such initial graph, recursively insert edges in lexicographic order between vertices of degree 2 in each possible way that does not create cycles of length smaller
  than $g$ until the graph is cubic. After each insertion of an edge $e$:
  \begin{description}
  \item [(i)] Check whether $e$ lies on an even 2-factor and if yes, backtrack.
    \end{description}

\end{description}

Implementing this basic algorithm and removing isomorphic graphs by using standard programs that remove isomorphic copies from lists of graphs, would give a very slow
program -- even if the computation of 2-factors, which would be the crucial part, was implemented very efficiently. It is essential that isomorphism rejection
is already used to decrease the number of graphs long before the completion to a cubic graph.

Different from our algorithm for g4-snarks, in this part we use {\em orderly generation} in the sense of Farad{\v z}ev and Read \cite{Fa76,Read78} for
isomorphism rejection.
In short (and admittedly not quite exact), orderly generation can be described as assigning a unique description of a graph as {\em the
  canonical description} of the graph -- typically by using string representations and defining the lexicographically minimal (or maximal) string as the
canonical one. Then the structures together with the corresponding string representations are constructed (typically in lexicographic order) and the graphs are accepted
if and only if the same graph can not be described by a smaller (resp.\ larger) string. We will describe the details for our algorithm, in which we will define the
lexicographically smallest string to be the canonical one.

Assume a graph with a 2-factor $F$ to be given. Let the odd cycle lengths in $F$ be $l^o_1\ge l^o_2\ge \dots \ge l^o_k$, and the even cycle lengths be
$l^e_1\ge l^e_2\ge \dots \ge l^e_j$. Note that the ordering of the even and odd lengths is opposite to the one defining the type. Then the string representing of the graph
starts with $S_1(F)=-l^o_1,\dots ,-l^o_k,0,-l^e_1,\dots ,-l^e_j,0$. The minus sign is to be able to later just take the lexicographically smallest string instead of
choosing some parts of the string to be maximal and others minimal.

In step b.) of the algorithm, this part $S_1(F_0)$ of the representation of all graphs constructed for this initial 2-factor $F_0$ is already fixed. To this end,
we can add a second bounding criterion to step c.):

 \begin{description}
  \item [(ii)] Check whether $e$ lies on a 2-factor $F'$ with $S_1(F')<S_1(F_0)$ and if yes, backtrack.
 \end{description}

 Of course (i) and (ii) can be tested together when searching for 2-factors. So the same 2-factor computing function can be used for bounding with respect to
 two different criteria: once that the resulting graph will be colourable and once that it will not be canonical. Although the total running time for $38$
 vertices is many years, the part for the initial 2-factor $F_0$ of type $(5,5,5,5,5,7,6)$ took less than 2 minutes, as even 2-factors or 2-factors $F'$ with
 $S_1(F')<S_1(F_0)$ turned up very early.

 \medskip

 For the second part of the representing string, we assume an ordered 2-factor $F$ and a legal labelling $L$ of $F$ to be given. The remaining entries of the
 string -- denoted as $S_2(F,L)$ --
 are the neighbours of $1,2,\dots ,n$ (in that order) that are not along edges of $F$. Examples of such representing strings are given in Figures~\ref{fig:blanusa0}
 and \ref{fig:blanusa}, where representing strings and corresponding labellings of one of the Blanusa snarks are given. If the graph is not yet 3-regular, we still use the
 notation $S_2(F,L)$, but it has some undefined entries. We denote the original labelling -- in fact the vertex numbers used in the algorithm -- as $L_0$.
 
\begin{figure}[tb]
	\centering
        \resizebox{0.5\textwidth}{0.5\textwidth}
       {
          \input{blanusa_0.tikz}
          }
	\caption{One of the Blanusa snarks drawn on the torus with a legal (but not canonical) labelling corresponding to a 2-factor of type 
          $(5,5,8)$. The corresponding string for this 2-factor and labelling is $-5,-5,0,-8,0,6,8,10,11,13,1,15,2,17,3,4,16,5,18,7,12,9,14$. The edges of the
          defining 2-factor are black and the edges in the remaining matching are green. In Figure~\ref{fig:blanusa}, the same graph with a canonical labelling is
          given.}
	\label{fig:blanusa0}
\end{figure}
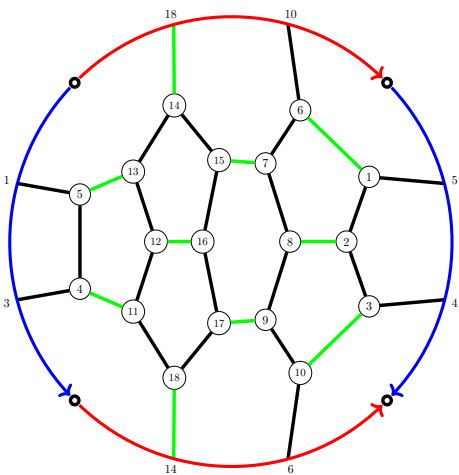
 
\begin{figure}[tb]
	\centering
        \resizebox{0.5\textwidth}{0.5\textwidth}
       {
          \input{blanusa.tikz}
          }
	\caption{One of the Blanusa snarks drawn on the torus with a canonical labelling. The canonical string is
          $-13,-5,0,0,6,8,10,7,13,1,4,2,16,3,15,17,5,18,11,9,12,14$. The edges of the defining canonical 2-factor are black and the edges in the remaining matching
          are green.}
	\label{fig:blanusa}
\end{figure}
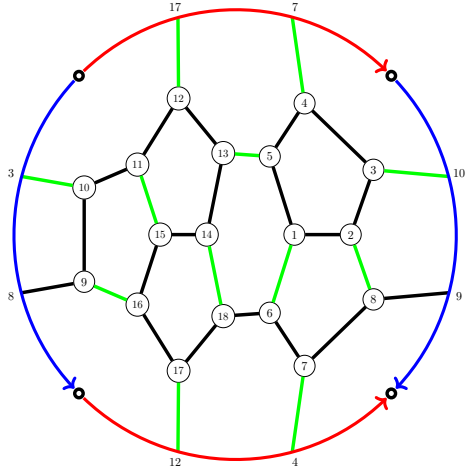

Also non-canonicity due to this part of the string can sometimes be detected before all edges are inserted. As long as not all edges are inserted, even for a given
legally labelled 2-factor, the second part of the representing string is not completely defined. Entries for vertices with degree 2 are not yet determined. We
say that a prefix of the second part is {\em complete}, if all entries of the prefix are already defined.  We can now 
add a third bounding criterion to step c.):

 \begin{description}
 \item [(iii)] Check whether there is a 2-factor $F'$ with $S_1(F')=S_1(F_0)$ and a labelling $L'$ of $F'$ so that a complete prefix of $S_2(F',L')$ is lexicographically smaller
   than the prefix of the same length of $S_2(F_0,L_0)$. If yes, backtrack.
 \end{description}

 Checking all legal labellings of a 2-factor seems to be an extremely time consuming task, but in most cases one can restrict the number of labellings to be
 tested a lot.  It turned out (see Table~\ref{tab:results2fc}) that -- at least inside the reach of the program -- almost all proper snarks on $n$ vertices have a
 2-factor of type $(5,n-5)$, which is due to our choice of $S_1()$ then also the canonical 2-factor.  For $n>10$ this means that for such graphs, trying to
 construct a minimal string, one has only 5 possibilities to give label 1 and two directions to continue the labeling. The neighbour of the vertex labelled 1
 must get label 6 -- otherwise the resulting string would not be minimal -- so except for the direction, the labelling of the second cycle is already determined
 by the choice for label 1. The choice of direction for the 5-cycle also implies which vertex gets label 2 -- and the neighbour of that vertex on the second
 cycle uniquely determines the direction. So in this case there are only 10 labellings to be checked and many of them can already be discarded after comparing
 two entries.  For other initial 2-factors, it can happen that due to the order of the cycles, also for another cycle than the first one all vertices must be
 tried as smallest vertices in that cycle, but these cases are often rejected early in the recursion.

 Nevertheless, applying the test always and in a straightforward way, would be much too expensive. Two key ideas were necessary to speed up things and 
 make the corresponding program fast enough to generate all proper snarks on 38 vertices:

 {\bf a.)} When the canonicity test fails, one can backtrack and therefore avoid the work in all subsequent branches of the recursion tree.
 The more non-2-factor edges have already been inserted, the smaller the size of the subtree that can be cut. So the tests get more expensive
 while at the same time the possible gain gets smaller. Not doing canonicity tests close to the end gives a small speedup: tests showed that not doing canonicity tests for the last 4
 edges (except of course when the graph is completed to a cubic graph) gives a total speedup of approximately $10\%$, which finally saved many CPU years.

 {\bf b.)} A larger speedup is obtained by the following method which was already used in an analogous way in \cite{minibaum} and \cite{permutationsnarks}. An
 exact speedup can not be given, as the method is an integral part of the implementation that can not just be removed or switched off. Assume that we have a
 graph $G$ for which $n_e$ non-2-factor edges have already been inserted and that a labelling $L'$ of a 2-factor $F'$ with $S_1(F')=S_1(F_0)$ -- with possibly $F'=F_0$,
 but $L'$ a different labelling than $L_0$ -- is checked.  If a complete prefix of $S_2(F',L')$ is smaller than a complete prefix of $S_2(F_0,L_0)$, due to
 (iii) we can backtrack, so $G$ has no successors in the recursion. If a complete prefix of $S_2(F',L')$ is larger than a complete prefix of $S_2(F_0,L_0)$,
 then this combination of $F',L'$ will result in a larger string for all graphs obtained by adding further edges and will never contradict the minimality of
 $S_1(F_0),S_2(F_0,L_0)$. So there is no need to test this combination for any successor of $G$. The same insight can be used for complete prefixes that are
 identical up to the first undefined position of $S_2(F',L')$: up to this position, no successor needs to compare the prefix -- it will be identical -- and the
 successor can start comparing the string when the vertex corresponding to the first undefined position has degree 3 -- and only from this position on.
 Building a data structure that allows to know exactly which 2-factors and labellings already established for smaller $n_e$ need to be evaluated further and
 from which position on, can result in comparing a prefix of $S_1(F_0),S_2(F_0,L_0)$ only once for the graph for which the prefix was first computed, instead of
 possibly millions of times for all its descendants.

 \bigskip

 \subsection{Strong Snarks}\label{sec:strong}
 
 There are several ways to even strengthen the requirements of uncolourability. We will now discuss one of the early requirements proposed by Jaeger \cite{Jaeger85}.
 For a graph $G$ and an edge $e$ of $G$ let $G_e$ denote the result of an {\em edge reduction} applied to $e$:
 the edge $e$ is removed and each vertex that has now degree 2 is removed and its former neighbours are connected. The inverse operation is called {\em edge insertion},
 so our edge doubling operation is a special case of edge insertion.
 Jaeger \cite{Jaeger85} proposed to call a snark $G$ a {\em strong snark} if for
 each edge $e$ the edge reduction applied to $e$ gives a class 2 graph.
 He did not require the snarks to have girth at least 5 and there are strong g4-snarks as well as
 strong proper snarks. For some conjectures -- e.g. the cycle double cover conjecture \cite{celmins} -- only strong snarks are possible smallest
 counterexamples -- but in fact often with even stronger requirements on the girth than just girth 5.

 So far, complete lists of strong snarks were obtained by filtering complete lists of snarks and in fact especially the complete list of proper snarks on 38 vertices produced by
 minisnark could be filtered very fast, as canonical 2-factors are encoded in the output. This is due to the following lemma that follows the ideas of the more general theorems
 30 in \cite{measures}  and 4 in \cite{criticalthesame}:

 \begin{lemma}\label{lem:strong}
   
A snark is a strong snark if and only if there is no 2-factor with exactly two odd cycles and an edge containing vertices of both odd cycles.

   \end{lemma}

 \begin{proof}

   $\Leftarrow:$ Let $G$ be a snark that is not strong. So there is an edge $e$, so that $G_e$ is colourable. Let such a colouring be given. If the new,
   formerly subdivided, edges $e_1,e_2$ in $G_e$ have the same colour $c_1$, let $c_2$ be an arbitrary colour different from $c_1$. Otherwise let $c_1,c_2$ be
   the colours of $e_1$, resp.\ $e_2$.  Let $F$ be the 2-factor defined by the edges with colours $c_1$ and $c_2$. If $e_1,e_2$ would be part of the same cycle
   in this 2-factor, this cycle with the edges subdivided again would be an even 2-factor of $G$ -- a contradiction. So they lie on different cycles and we get
   a 2-factor $F'$ of $G$ with $e$ connecting the only two odd cycles in $F'$.

   $\Rightarrow:$ Let a snark $G$ have an edge $e$ connecting the only two odd cycles in a 2-factor, in $G_e$. Then these two cycles become
   even in $G_e$, so $G_e$ is colourable and $G$ is therefore not a strong snark.

 \end{proof}

 Even without using the special structure of the coded graphs, applying Lemma~\ref{lem:strong} for filtering the list of all 1.051.945.967 proper snarks on 38
 vertices took only about 14 minutes on an Intel(R) Core(TM) i7-9700 CPU restricted to 3 GHz -- the time for decompressing the file of input graphs (about 3.3
 minutes) not counted. When using the defining 2-factors that can easily be reconstructed from the labeling, it took even less time. E.g. the $99,998\%$ of the 
 graphs with the canonical 2-factor of type $(5,33)$ (but also many others) could immediately be decided not to be strong.

 But as {\em minisnark} constructs all 2-factors anyway (unless it backtracks due to an even 2-factor or a 2-factor contradicting canonicity), the basic algorithm
 of minisnark could  easily be extended as follows to get one step further for proper snarks than just by filtering:

 Initial 2-factors with only two cycles, which are then necessarily odd,  are not used at all. As there would always be edges between them, the snarks constructed from
 them would never be strong.  When the initial 2-factor has next to even cycles exactly 2 odd cycles, all edges between the odd cycles are forbidden for the construction.  Whenever a new
 2-factor with exactly 2 odd cycles is constructed after a new edge has been inserted, it is tested whether there are edges between the two cycles and if yes,
 the program backtracks. If there are no such edges, for the descendants of this subgraph, all edges between these two new cycles are forbidden for the construction.

 This algorithm can be applied for strong g4-snarks as well as for strong proper snarks, but as constructing g4-snarks on $n+2$ vertices is relatively fast -- if all proper snarks
 on up to $n$ vertices are available -- for strong g4-snarks it is faster to generate all g4-snarks and filter them for strong ones.
 
 \medskip
 
 If all snarks on $n$ vertices are available, there is also another method to generate all strong proper snarks on $n+2$ vertices. The following lemma is
 already proven in \cite{removableedges4cyc} as Lemma~10(b), but as the proof is short, we will give an independent proof here.

\begin{lemma}\label{lem:reduce2}
  (Andersen, Fleischner, Jackson \cite{removableedges4cyc})

  Let $G$ be a cyclically 4-connected cubic graph and $C_0$ a component after removing a cyclic 4-cut $C$ with $|C_0|>4$ and so that no proper subset of $C_0$ is a component after removing a cyclic 4-cut.
  If $e\subset C_0$, then applying an edge reduction to $e$ gives a cyclically 4-connected graph.
 
 \end{lemma}

 \begin{proof}

   Due to Observation~\ref{obs:reduce}, it is sufficient to prove that $e$ is not contained in a cyclic 4-cut. Let $C_1$ be the other component after removing $C$. If $e$ is contained in a cyclic 4-cut $D$,
   then $D$ has components $D_0$, $D_1$ and none of them is a proper subset of $C_0$, as $C_0$ is minimal, and none is a proper subset of $C_1$, as $e\in D$ and both endpoints of $e$
   are in $C_0$, but in different components after removing $D$. So we have a partition $C_0\cap D_0$, $C_0\cap D_1$, $C_1\cap D_0$, and $C_1\cap D_1$ of the vertex set and $8$ edges between these sets.
   None of the sets contains just one vertex, as in that case three edges would leave the set and at least two of them would belong to the same cut -- but edges in minimal cuts do not share vertices.
   As $|C_0|>4$, at least one of $C_0\cap D_0$, $C_0\cap D_1$ ( w.l.o.g.\  $C_0\cap D_0$) has more than 2 vertices, so more than 4 edges leave  $C_0\cap D_0$: if  $C_0\cap D_0$ contains a cycle this is due to the fact
   that $C_0$ is minimal and otherwise because more than 2 vertices in a component without a cycle imply more than 4 outgoing edges. Counting edges as outgoing for both sides, at most
   $11$ edges can leave the other 3 sets, so for at least one set there are at most 3 edges leaving it. If it would be a trivial cut, there would be two edges of a minimal cut sharing a vertex, so these 3 edges form a cyclic cut
   -- a contradiction.

 \end{proof}

 \begin{lemma}

   All strong snarks on $n+2$ vertices can be constructed from (not necessarily strong) snarks on $n$ vertices by edge insertion.

 \end{lemma}

 \begin{proof}

  If a strong snark $G$ on $n+2$ vertices has girth $4$, this follows from Lemma~\ref{lem:weak} even without the requirement that $G$ is strong. If the girth is
  at least $5$, then each minimal component has more than $4$ vertices, so Lemma~\ref{lem:reduce2} implies that an edge reduction applied to an edge in such a minimal component
  produces a cyclically 4-connected graph, which is class 2 as $G$ is strong.

 \end{proof}

 While strong snarks with girth 5 can not be constructed by edge insertion only from strong snarks with fewer vertices, the situation is again much nicer for
 girth 4. The following lemma also explains why strong snarks with girth 4 occur later than the first ones with girth 5.

 \begin{lemma}\label{lem:allstrong}

   All strong snarks with girth $4$ on $n+2$ vertices can be constructed from strong snarks on $n$ vertices by edge-doubling.
   
   Furthermore, a snark constructed by applying the reverse of an edge doubling operation to a strong snark, is strong, so it is even true that by edge doubling,
   strong snarks can only be constructed from strong snarks.

 \end{lemma}

 \begin{proof}

   The fact that strong snarks with girth 4 (just like those that are not strong) can be constructed from smaller snarks by edge doubling follows from
   Lemma~\ref{lem:weak}. It remains to be shown that the snark it was constructed from is strong, or equivalently that the inverse of an edge doubling operation
   applied to a strong snark always results in a strong snark. Let $G_{n+2}$ be a strong snark on $n+2$ vertices and $G_n$ be a snark that is the result of an edge
   reduction applied to an edge of the 4-cycle $Z$. We have to show that $G_n$ is strong, that is: that for an edge $e$ of $G_n$, the graph $(G_n)_e$ is
   class 2.  We use the notation of Figure~\ref{fig:weakconstruct}.  If $e$ does not contain a vertex of $e_2$, $(G_n)_e$ contains $(G_{n+2})_e-Z$ as a
   subgraph, but $(G_{n+2})_e$ is class 2 as $G_{n+2}$ is strong and as it still contains $Z$, we can apply Lemma~\ref{lem:deco} which implies that
   $(G_{n+2})_e-Z$ is class 2, so also $(G_n)_e$. If $e=e_2$, then $G_{n+2}-Z$ is a subgraph of $(G_n)_e$, so again we have that $(G_n)_e$ is class 2.

   The last remaining case is that $e$ is one of the edges containg exactly one vertex of $e_2$ -- w.l.o.g.\ $e=e_1$. This case is depicted in
   Figure~\ref{fig:weaksnark} with the removed edge $e$ shown as a dashed line. In this case $(G_n)_e$ is the graph obtained from $(G_{n+2})_e$ by contracting a
   triangle to a vertex, which preserves the colour class -- so it is class 2 as $(G_{n+2})_e$ is class 2.
   
\begin{figure}[tb]
	\centering
	\includegraphics[width=0.7\textwidth]{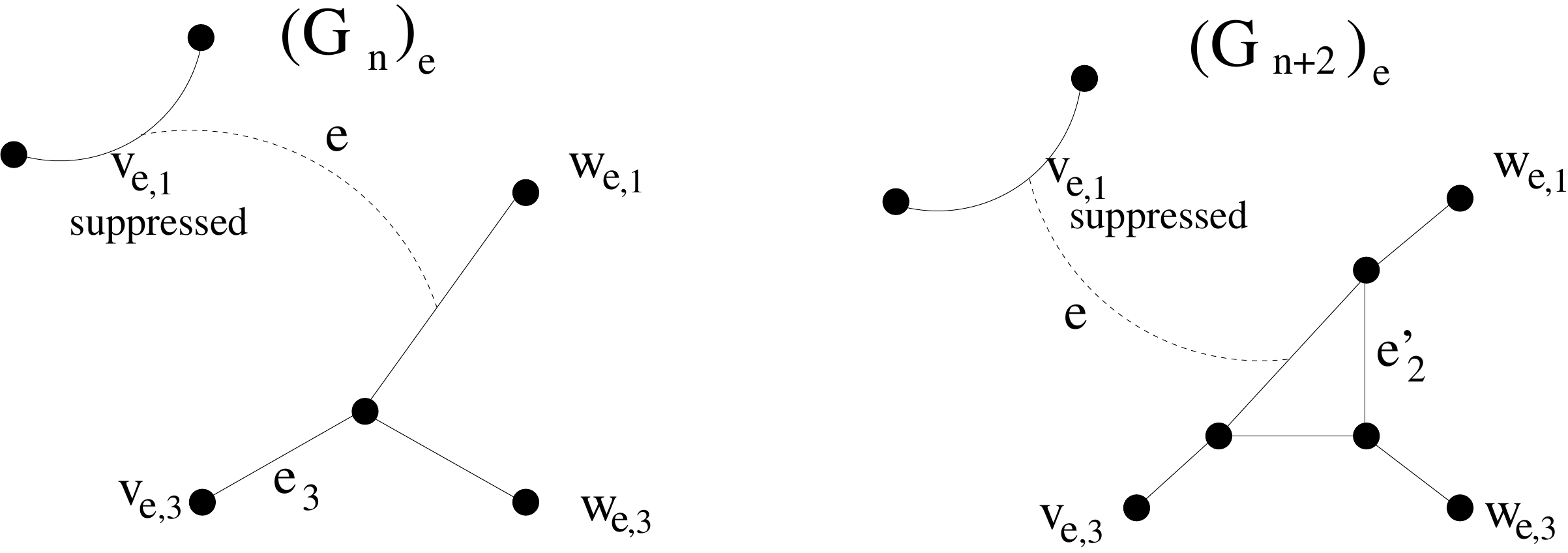}
	\caption{The graphs $(G_n)_e$ and $(G_{n+2})_e$ for $e=e_1$ with the notation of Figure~\ref{fig:weakconstruct}.}
	\label{fig:weaksnark}
\end{figure}

   \end{proof}

\section{Results}

In Table~\ref{tab:results} we give the numbers of g4-snarks and proper snarks found by the programs {\em tetration} and {\em minisnark} implementing the
algorithms described in the previous parts. The programs can be obtained from the authors and from the \verb+arXiv.org+ page of the article, where they are stored as ancillary files.
The lists of all proper snarks and lists of some g4-snarks can be downloaded from the House of Graphs \cite{HoG2}.

\vspace*{0.7cm}

\begin{table}
  \begin{center}
    \begin{minipage}{16cm}
    \begin{tabular}{|l|l|l|l|l|l|}
      \hline
 number of   & girth 4 snarks & proper snarks  & proper snarks & proper snarks &  cyclically \\
 vertices &  (girth$=4$)  & (girth$\ge 5$)  & with girth 6 & with girth 7 &  5-connected  \\
    & &  &    &  &   proper snarks  \\
  \hline
  10 &  0 ** &  1 ** & 0 ** & 0 ** & 1 **  \\
  \hline
  12 & 0 **  & 0 **  & 0 ** & 0 ** &  0 ** \\
  \hline
  14 &  0 ** & 0 **  &0 **  &0 **  &  0 ** \\
  \hline
  16 & 0 **  & 0 **  &0 **  &0 **  & 0 **  \\
  \hline
  18 &  0  ** & 2 **  & 0  ** &   0 ** & 0 **  \\
  \hline
  20 &  0  ** & 6  ** &  0 ** &   0 ** &  1 ** \\
  \hline
  22 & 11   ** & 20  ** &  0 ** &  0 **  &  2 ** \\
  \hline
  24 & 117   ** & 38  ** & 0  ** &   0 ** &  2 ** \\
  \hline
  26 & 1,017   ** &  280 ** & 0  ** &   0 ** & 10 **  \\
  \hline
  28 & 9,617   ** & 2,900  ** & 1   ** &   0 ** & 75 **  \\
  \hline
  30 &  111,455  ** & 28,399  ** & 1  ** &   0 ** &  509 ** \\
  \hline
  32 & 1,471,891   ** &  293,059 ** & 0   ** &   0 ** &  2,953  ** \\
  \hline
  34 &  21,453,366  * &  3,833,587 * &  0 * &   0 * & 19,935  * \\
  \hline
  36 & 344,732,184   * & 60,167,732  *  & 1  * &   0 * &  180,612 * \\
  \hline
  38 & 6,090,271,932  & 1,051,945,967   & 39   * &   0 * &  1,925,620  \\
  \hline
  40 & 113,585,516,609  &   & 276  &   0 * &   \\
  \hline
  42 &   &   &  &   0 * &   \\
  \hline
    \end{tabular}
     \end{minipage}
\caption{Numbers of snarks.  In the table, numbers confirmed by snarkhunter are marked with a * and numbers
  confirmed by snarkhunter and minibaum are marked **. The other numbers were produced by the programs described in this article.
  Of course for small numbers of vertices, the numbers of snarks were already known long before these
programs.} \label{tab:results}
\end{center}
\end{table}

\begin{table}
  \begin{center}
\hfill\begin{minipage}{6cm}
  \begin{tabular}{|l|l|}
    \hline
  2-factor & number of graphs \\
  \hline
(5,11,22)  &  3 \\
\hline
(5,13,20)  &  5 \\
\hline
(5,15,18)  &  5 \\
\hline
(5,17,16)  &  6 \\
\hline
(5,19,14)  &  32 \\
\hline
(5,21,12)  &  27 \\
\hline
(5,21,6,6)  &  3 \\
\hline
(5,23,10)  &  40 \\
\hline
(5,25,8)  &  804 \\
\hline
(5,27,6)  &  7028 \\
\hline
(5,33) &  1,051,924,110 \\
\hline
(5,5,11,17)  &  54 \\
\hline
(5,5,13,15)  &  24 \\
\hline
(5,5,5,17,6)  &  1 \\
\hline
(5,5,5,23)  &  3,442 \\
\hline
(5,5,7,21)  &  942 \\
\hline
(5,5,9,13,6)  &  1 \\
\hline
(5,5,9,19)  &  575 \\
\hline
(5,7,13,13)  &  1 \\
\hline
(5,7,9,17)  &  6 \\
\hline
(5,9,9,15)  &  3 \\
\hline
(7,17,14)  &  1 \\
\hline
(7,21,10)  &  3 \\
\hline
(7,23,8)  &  1 \\
\hline
(7,25,6)  &  179 \\
\hline
\end{tabular}
\end{minipage}\hfill
\begin{minipage}{6cm}
  \begin{tabular}{|l|l|}
    \hline
  2-factor & number of graphs \\
    \hline

(7,31)  &  1,636 \\
\hline
(9,15,14)  &  2 \\
\hline
(9,17,12)  &  3 \\
\hline
(9,17,6,6)  &  4 \\
\hline
(9,19,10)  &  2 \\
\hline
(9,21,8)  &  20 \\
\hline
(9,23,6)  &  496 \\
\hline
(9,29)  &  2,029 \\
\hline
(11,11,16)  &  1 \\
\hline
(11,17,10)  &  1 \\
\hline
(11,19,8)  &  3 \\
\hline
(11,21,6)  &  29 \\
\hline
(11,27)  &  1,266 \\
\hline
(13,13,12)  &  11 \\
\hline
(13,13,6,6)  &  5 \\
\hline
(13,15,10)  &  2 \\
\hline
(13,17,8)  &  41 \\
\hline
(13,19,6)  &  174 \\
\hline
(13,25)  &  2,436 \\
\hline
(15,15,8)  &  1 \\
\hline
(15,17,6)  &  2 \\
\hline
(15,23)  &  346 \\
\hline
(17,21)  &  68 \\
\hline
(19,19)  &  93 \\
\hline
\end{tabular}
\end{minipage}
\caption{Numbers of snarks for the various canonical 2-factors of the proper snarks on 38 vertices.} \label{tab:results2fc}
\end{center}
\end{table}

\begin{table}
  \begin{center}
    \begin{minipage}{6cm}
    \begin{tabular}{|l|l|l|}
      \hline
 number of   & strong  & strong \\ 
 vertices &  g4-snarks &  proper snarks \\
    & (girth$=4$) &  (girth$\ge 5$) \\
  \hline
  34 & 0*   &   7*  \\
  \hline
  36 & 33*  &  25*  \\
  \hline
  38 &  407  &  298*   \\
  \hline
  40 & 4,553   &   7,654  \\
  \hline
    42 & 79,731   &  $\ge$ 142,130   \\
  \hline
    \end{tabular}
    \end{minipage}
\caption{Numbers of strong snarks. There are no strong snarks with fewer than 34 vertices and all strong snarks in the table have a cyclic 4-cut. Numbers marked
  with a * were already computed before by filtering the output of snarkhunter \cite{snarkapp} (36 vertices) and by applying edge insertions (38 vertices).
  Nevertheless at the time the list of strong snarks on 38 vertices was produced, Lemma~\ref{lem:allstrong} was not yet known, so the list was presented as
  possibly not being complete. Note that there are more strong g4-snarks than strong proper snarks on 36 and 38 vertices, but fewer on 40 and 42 vertices. The
  lower bound for the number of strong proper snarks on 42 vertices comes from only those that can be obtained from strong snarks on 40 vertices by edge
  insertion and is therefore most likely far from the total number of strong proper snarks. E.g. for 40 vertices about 80$\%$ of the strong proper snarks can be
  obtained this way from strong snarks on 38 vertices.} \label{tab:resultsstrong}
\end{center}
\end{table}

Among the proper snarks on 38 vertices, $99,998\%$ had a 2-factor of type $(5,33)$, so they were generated for this initial 2-factor. Nevertheless, the number of different
canonical 2-factors grows. While for 24 vertices the only canonical 2-factor had type $(5,19)$ and for 26 vertices only the 2-factors of type $(5,21)$ and $(13,13)$ (one graph) were
canonical, for $38$ vertices already 49 different 2-factors occur as canonical ones. Up to 38 vertices, only canonical 2-factors with at most three 5-cycles occur, but it
is known \cite{2fac_5cyc} that for each $k$ there are proper snarks where each 2-factor -- so also a canonical one -- contains at least $k$ 5-cycles. This also
shows that the number of cycles in a canonical 2-factor of a proper snark can not be bounded by a constant.  The proper snarks with girth 6 in the list all have a canonical 2-factor of type
$(7,|V|-7)$, except for the snark on 36 vertices which has a canonical 2-factor of type $(9,27)$.

The time necessary for the proper snarks on $38$
vertices was a bit less than 60 CPU years on AMD EPYC 7552 (Rome @ 2.2 GHz) processors. For girth 6 and $40$ vertices it was a bit less than 40 CPU years. The
test for the existence of snarks with girth 7 on 42 vertices took slightly more than 5 CPU years. The generation of all g4-snarks on 40 vertices was run on different machines,
but normalized to the same AMD EPYC 7552 (Rome @ 2.2 GHz) processors took a bit less than 3.5 CPU years.

\section{Final remarks}

When computer algorithms and their results are presented, there is always one big step that is not reflected in the article: the translation from the algorithm
to the program that was actually used to compute the results. Of course, the program code must be available too, but really checking every detail of the program
code is very hard. During the translation, many steps that are just sketched in the article must be implemented in detail and, in order to be as efficient as
possible, sometimes in a way that is not straightforward. An example: when the algorithm description says {\em ``for girth at least $g$, never connect vertex
  $v$ to a vertex $w$ at distance smaller than $g-1$''}, this seems a straightforward test to be done. In order to do this efficiently, for all possible $w$, in
fact a set of vertices at distance at least $g-1$ is computed. To do also this computation as fast as possible, for each vertex next to the set of adjacent
vertices also the set of vertices at distances 2 and 3 are stored -- or to be exact: the sets of vertices with degree 2 at distances 2 and 3. These sets have to
be updated whenever an edge is added or removed. These details of an implementation are not innovative or complicated enough to justify a description in an article,
but can nevertheless be the cause for errors.

To this end programs need to be thoroughly tested -- parts as well as the whole program. Independent tests of computational results are very valuable. {\bf If}
there is an error in an implementation, the chance is very small that a possible error in the implementation of an independent algorithm has the same impact. This is maybe
even more reliable than two referees looking at a theoretical proof, as the approach of the two referees is most likely quite similar and there is no reason to
believe that they detect different kinds of small hidden errors. 

The lists confirmed by two or even three independent programs can therefore be considered extremely reliable.
Based on the agreement of these large lists, also the lists only constructed by minisnark or tetration can be considered as at least as reliable as a well refereed theoretical
proof.  Nevertheless, an independent new approach confirming these lists would be valuable -- for new ideas in the algorithm as well as for confirming the
numbers. Of course allowing the next step -- that is e.g. proper snarks on 40 vertices -- would be even better. As the speed of computers is not increasing
fast any more, an algorithm would have to be approximately 20 times faster than the ones presented in this paper. Alternatively maybe the use of GPUs
can make the next step possible.

\bigskip

\section{ Acknowledgement}

A note of the first author: I want to express my gratitude to my late PhD supervisor Andreas Dress -- not only for sharing his knowledge of which I could only absorb a
small part, but also for his support and for being open for {\bf all} kinds of mathematical, chemical, or biological problems. I tried to learn from that too\dots

The by far major part of the computational resources (Stevin Supercomputer Infrastructure) and services used in this work were provided by the VSC (Flemish Supercomputer
Center), funded by Ghent University, FWO and the Flemish Government – department EWI.

\bibliography{/home/gbrinkma/schreib/literatur.bib}

\end{document}

%% file: blanusa_0.tikz
\begin{tikzpicture}[scale=0.065]
\def\vertexscale{1.00}
\def\labelscale{1.00}
\node [circle,black,draw,scale=\vertexscale] (1) at (62.62729,28.84362) {1};
\node [circle,black,draw,scale=\vertexscale] (2) at (52.33802,0.00000) {2};
\node [circle,black,draw,scale=\vertexscale] (3) at (62.62729,-28.84362) {3};
\node [circle,black,draw,scale=\vertexscale] (4) at (-68.32644,-20.97256) {4};
\node [circle,black,draw,scale=\vertexscale] (5) at (-68.32644,20.97256) {5};
\node [circle,black,draw,scale=\vertexscale] (6) at (31.44361,58.46809) {6};
\node [circle,black,draw,scale=\vertexscale] (7) at (15.69310,34.81163) {7};
\node [circle,black,draw,scale=\vertexscale] (8) at (26.82647,0.00000) {8};
\node [circle,black,draw,scale=\vertexscale] (9) at (15.69310,-34.81163) {9};
\node [circle,black,draw,scale=0.9*\vertexscale] (10) at (31.44361,-58.46809) {10};
\node [circle,black,draw,scale=0.9*\vertexscale] (11) at (-44.22767,-31.16284) {11};
\node [circle,black,draw,scale=0.9*\vertexscale] (12) at (-33.91844,-0.00000) {12};
\node [circle,black,draw,scale=0.9*\vertexscale] (13) at (-44.22767,31.16284) {13};
\node [circle,black,draw,scale=0.9*\vertexscale] (14) at (-25.59232,60.62887) {14};
\node [circle,black,draw,scale=0.9*\vertexscale] (15) at (-5.35545,36.29170) {15};
\node [circle,black,draw,scale=0.9*\vertexscale] (16) at (-12.79022,-0.00000) {16};
\node [circle,black,draw,scale=0.9*\vertexscale] (17) at (-5.35545,-36.29170) {17};
\node [circle,black,draw,scale=0.9*\vertexscale] (18) at (-25.59232,-60.62887) {18};
\tkzDefPoint(-96.59258,25.88190){19}
\tkzDefPoint(-70.71068,-70.71068){20}
\tkzDefPoint(-96.59258,-25.88190){21}
\tkzDefPoint(-25.88190,-96.59258){22}
\tkzDefPoint(25.88190,-96.59258){23}
\tkzDefPoint(96.59258,25.88190){24}
\tkzDefPoint(96.59258,-25.88190){25}
\tkzDefPoint(-25.88190,96.59258){26}
\tkzDefPoint(25.88190,96.59258){27}
\tkzDefPoint(-70.71068,70.71068){28}
\tkzDefPoint(70.71068,70.71068){29}
\tkzDefPoint(70.71068,-70.71068){30}
\draw [black, line width=1mm] (1) to (2);
\draw [green, line width=1mm] (1) to (6);
\draw [black, line width=1mm] (1) to (24);
\node [draw=none,black,fill=none,scale=\labelscale] () at (101.42221,27.17600) {5};
\draw [black, line width=1mm] (2) to (3);
\draw [green, line width=1mm] (2) to (8);
\draw [black, line width=1mm] (3) to (25);
\node [draw=none,black,fill=none,scale=\labelscale] () at (101.42221,-27.17600) {4};
\draw [green, line width=1mm] (3) to (10);
\draw [black, line width=1mm] (4) to (21);
\node [draw=none,black,fill=none,scale=\labelscale] () at (-101.42221,-27.17600) {3};
\draw [black, line width=1mm] (4) to (5);
\draw [green, line width=1mm] (4) to (11);
\draw [black, line width=1mm] (5) to (19);
\node [draw=none,black,fill=none,scale=\labelscale] () at (-101.42221,27.17600) {1};
\draw [green, line width=1mm] (5) to (13);
\draw [black, line width=1mm] (6) to (7);
\draw [black, line width=1mm] (6) to (27);
\node [draw=none,black,fill=none,scale=\labelscale] () at (27.17600,101.42221) {10};
\draw [green, line width=1mm] (7) to (15);
\draw [black, line width=1mm] (7) to (8);
\draw [black, line width=1mm] (8) to (9);
\draw [black, line width=1mm] (9) to (10);
\draw [green, line width=1mm] (9) to (17);
\draw [black, line width=1mm] (10) to (23);
\node [draw=none,black,fill=none,scale=\labelscale] () at (27.17600,-101.42221) {6};
\draw [black, line width=1mm] (11) to (12);
\draw [black, line width=1mm] (11) to (18);
\draw [black, line width=1mm] (12) to (13);
\draw [green, line width=1mm] (12) to (16);
\draw [black, line width=1mm] (13) to (14);
\draw [green, line width=1mm] (14) to (26);
\node [draw=none,black,fill=none,scale=\labelscale] () at (-27.17600,101.42221) {18};
\draw [black, line width=1mm] (14) to (15);
\draw [black, line width=1mm] (15) to (16);
\draw [black, line width=1mm] (16) to (17);
\draw [black, line width=1mm] (17) to (18);
\draw [green, line width=1mm] (18) to (22);
\node [draw=none,black,fill=none,scale=\labelscale] () at (-27.17600,-101.42221) {14};
\tkzDefPoint(-68.55786,72.79986){A}
\tkzDefPoint(68.55786,72.79986){B}
\tkzDefPoint(0.0,0.0){C}
\tkzDrawArc[<-,line width=0.9mm, red](C,B)(A)
\tkzDefPoint(72.79986,68.55786){A}
\tkzDefPoint(72.79986,-68.55786){B}
\tkzDefPoint(0.0,0.0){C}
\tkzDrawArc[<-,line width=0.9mm, blue](C,B)(A)
\tkzDefPoint(68.55786,-72.79986){A}
\tkzDefPoint(-68.55786,-72.79986){B}
\tkzDefPoint(0.0,0.0){C}
\tkzDrawArc[->,line width=0.9mm, red](C,B)(A)
\tkzDefPoint(-72.79986,-68.55786){A}
\tkzDefPoint(-72.79986,68.55786){B}
\tkzDefPoint(0.0,0.0){C}
\tkzDrawArc[->,line width=0.9mm, blue](C,B)(A)
\node [black,circle,draw,fill=white,scale=0.75,line width=1mm] (20) at (-70.71068,-70.71068) {};
\node [black,circle,draw,fill=white,scale=0.75,line width=1mm] (28) at (-70.71068,70.71068) {};
\node [black,circle,draw,fill=white,scale=0.75,line width=1mm] (29) at (70.71068,70.71068) {};
\node [black,circle,draw,fill=white,scale=0.75,line width=1mm] (30) at (70.71068,-70.71068) {};
\end{tikzpicture}

%% file: blanusa.tikz
\begin{tikzpicture}[scale=0.065]
\def\vertexscale{1.00}
\def\labelscale{1.00}
\node [circle,black,draw,scale=\vertexscale] (1) at (26.82647,0.00000) {1};
\node [circle,black,draw,scale=\vertexscale] (2) at (52.33802,0.00000) {2};
\node [circle,black,draw,scale=\vertexscale] (3) at (62.62729,28.84362) {3};
\node [circle,black,draw,scale=\vertexscale] (4) at (31.44361,58.46809) {4};
\node [circle,black,draw,scale=\vertexscale] (5) at (15.69310,34.81163) {5};
\node [circle,black,draw,scale=\vertexscale] (6) at (15.69310,-34.81163) {6};
\node [circle,black,draw,scale=\vertexscale] (7) at (31.44361,-58.46809) {7};
\node [circle,black,draw,scale=\vertexscale] (8) at (62.62729,-28.84362) {8};
\node [circle,black,draw,scale=\vertexscale] (9) at (-68.32644,-20.97256) {9};
\node [circle,black,draw,scale=0.9*\vertexscale] (10) at (-68.32644,20.97256) {10};
\node [circle,black,draw,scale=0.9*\vertexscale] (11) at (-44.22767,31.16284) {11};
\node [circle,black,draw,scale=0.9*\vertexscale] (12) at (-25.59232,60.62887) {12};
\node [circle,black,draw,scale=0.9*\vertexscale] (13) at (-5.35545,36.29170) {13};
\node [circle,black,draw,scale=0.9*\vertexscale] (14) at (-12.79022,-0.00000) {14};
\node [circle,black,draw,scale=0.9*\vertexscale] (15) at (-33.91844,-0.00000) {15};
\node [circle,black,draw,scale=0.9*\vertexscale] (16) at (-44.22767,-31.16284) {16};
\node [circle,black,draw,scale=0.9*\vertexscale] (17) at (-25.59232,-60.62887) {17};
\node [circle,black,draw,scale=0.9*\vertexscale] (18) at (-5.35545,-36.29170) {18};
\tkzDefPoint(-96.59258,25.88190){19}
\tkzDefPoint(-70.71068,-70.71068){20}
\tkzDefPoint(-96.59258,-25.88190){21}
\tkzDefPoint(-25.88190,-96.59258){22}
\tkzDefPoint(25.88190,-96.59258){23}
\tkzDefPoint(96.59258,25.88190){24}
\tkzDefPoint(96.59258,-25.88190){25}
\tkzDefPoint(-25.88190,96.59258){26}
\tkzDefPoint(25.88190,96.59258){27}
\tkzDefPoint(-70.71068,70.71068){28}
\tkzDefPoint(70.71068,70.71068){29}
\tkzDefPoint(70.71068,-70.71068){30}
\draw [black, line width=1mm] (1) to (2);
\draw [green, line width=1mm] (1) to (6);
\draw [black, line width=1mm] (1) to (5);
\draw [black, line width=1mm] (2) to (3);
\draw [green, line width=1mm] (2) to (8);
\draw [black, line width=1mm] (3) to (4);
\draw [green, line width=1mm] (3) to (24);
\node [draw=none,black,fill=none,scale=\labelscale] () at (101.42221,27.17600) {10};
\draw [black, line width=1mm] (4) to (5);
\draw [green, line width=1mm] (4) to (27);
\node [draw=none,black,fill=none,scale=\labelscale] () at (27.17600,101.42221) {7};
\draw [green, line width=1mm] (5) to (13);
\draw [black, line width=1mm] (6) to (7);
\draw [black, line width=1mm] (6) to (18);
\draw [green, line width=1mm] (7) to (23);
\node [draw=none,black,fill=none,scale=\labelscale] () at (27.17600,-101.42221) {4};
\draw [black, line width=1mm] (7) to (8);
\draw [black, line width=1mm] (8) to (25);
\node [draw=none,black,fill=none,scale=\labelscale] () at (101.42221,-27.17600) {9};
\draw [black, line width=1mm] (9) to (21);
\node [draw=none,black,fill=none,scale=\labelscale] () at (-101.42221,-27.17600) {8};
\draw [black, line width=1mm] (9) to (10);
\draw [green, line width=1mm] (9) to (16);
\draw [green, line width=1mm] (10) to (19);
\node [draw=none,black,fill=none,scale=\labelscale] () at (-101.42221,27.17600) {3};
\draw [black, line width=1mm] (10) to (11);
\draw [black, line width=1mm] (11) to (12);
\draw [green, line width=1mm] (11) to (15);
\draw [green, line width=1mm] (12) to (26);
\node [draw=none,black,fill=none,scale=\labelscale] () at (-27.17600,101.42221) {17};
\draw [black, line width=1mm] (12) to (13);
\draw [black, line width=1mm] (13) to (14);
\draw [green, line width=1mm] (14) to (18);
\draw [black, line width=1mm] (14) to (15);
\draw [black, line width=1mm] (15) to (16);
\draw [black, line width=1mm] (16) to (17);
\draw [black, line width=1mm] (17) to (18);
\draw [green, line width=1mm] (17) to (22);
\node [draw=none,black,fill=none,scale=\labelscale] () at (-27.17600,-101.42221) {12};
\tkzDefPoint(-68.55786,72.79986){A}
\tkzDefPoint(68.55786,72.79986){B}
\tkzDefPoint(0.0,0.0){C}
\tkzDrawArc[<-,line width=0.9mm, red](C,B)(A)
\tkzDefPoint(72.79986,68.55786){A}
\tkzDefPoint(72.79986,-68.55786){B}
\tkzDefPoint(0.0,0.0){C}
\tkzDrawArc[<-,line width=0.9mm, blue](C,B)(A)
\tkzDefPoint(68.55786,-72.79986){A}
\tkzDefPoint(-68.55786,-72.79986){B}
\tkzDefPoint(0.0,0.0){C}
\tkzDrawArc[->,line width=0.9mm, red](C,B)(A)
\tkzDefPoint(-72.79986,-68.55786){A}
\tkzDefPoint(-72.79986,68.55786){B}
\tkzDefPoint(0.0,0.0){C}
\tkzDrawArc[->,line width=0.9mm, blue](C,B)(A)
\node [black,circle,draw,fill=white,scale=0.75,line width=1mm] (20) at (-70.71068,-70.71068) {};
\node [black,circle,draw,fill=white,scale=0.75,line width=1mm] (28) at (-70.71068,70.71068) {};
\node [black,circle,draw,fill=white,scale=0.75,line width=1mm] (29) at (70.71068,70.71068) {};
\node [black,circle,draw,fill=white,scale=0.75,line width=1mm] (30) at (70.71068,-70.71068) {};
\end{tikzpicture}